\documentclass[11pt,a4paper]{amsart}		

\usepackage[applemac]{inputenc}					

\usepackage{lmodern}						
\usepackage[english]{babel}					
\usepackage[weather]{ifsym}					

\usepackage{hyperref,enumitem,soul}

\hypersetup{
	colorlinks,
	linkcolor={red!50!black},
	citecolor={blue!50!black},
	urlcolor={blue!80!black}
}

\usepackage{tikz}
\usetikzlibrary{calc, svg.path, patterns}
\usepackage{todonotes}

\usepackage{amsmath,amsfonts,amssymb,amsthm}

\usepackage[left=3cm,right=3cm,top=3cm,bottom=3cm]{geometry}

\usepackage{mathtools}						

\theoremstyle{plain}
\newtheorem{theorem}{Theorem}
\newtheorem{prop}[theorem]{Proposition}
\newtheorem{proposition}[theorem]{Proposition}
\newtheorem{cor}[theorem]{Corollary}
\newtheorem{corollary}[theorem]{Corollary}
\newtheorem{lemma}[theorem]{Lemma}
\newtheorem*{conj}{Open Problem}
\theoremstyle{definition}

\theoremstyle{remark}
\newtheorem{rmk}[theorem]{Remark}

\DeclareMathOperator{\tr}{\mathrm{trace}}

\newcommand{\R}{\mathbb{R}}					
\newcommand{\N}{\mathbb{N}}					

\newcommand{\distr}{D}			
\renewcommand{\theta}{\vartheta}			
\renewcommand{\epsilon}{\varepsilon}		

\newcommand{\eps}{{\varepsilon}}

\title[Abnormal minimizers for rank 2 sub-Riemannian structures]{On the regularity of abnormal minimizers for rank 2 sub-Riemannian structures}

\date{\today}
\author{D.\ Barilari}
\address{Univ.\ Paris Diderot, Institut de Math\'ematiques de Jussieu-Paris Rive Gauche, CNRS, Sorbonne Universit\'e.
B\^atiment Sophie-Germain, case 7012, 75205 Paris cedex 13, France}
\email{davide.barilari@imj-prg.fr}
\author{Y.\ Chitour}
\address{Universit\'e Paris-Sud, L2S, CentraleSup\'elec, Universit\'e
Paris-Saclay, Gif-sur-Yvette, France}
{\email{yacine.chitour@l2s.centralesupelec.fr}
\author{F.\ Jean}
\address{Unit\'e de Math\'ematiques Appliqu\'ees, ENSTA ParisTech, Universit\'e
Paris-Saclay, 91120 Palaiseau, France}
\email{frederic.jean@ensta-paristech.fr}
\author{D.\ Prandi}
\address{CRNS, L2S, CentraleSup\'elec, Universit\'e
Paris-Saclay, Gif-sur-Yvette, France}
\email{dario.prandi@l2s.centralesupelec.fr}
\author{M.\ Sigalotti}
\address{Inria \& Laboratoire Jacques-Louis Lions, Sorbonne Universit\'e,
75005 Paris, France}
\email{mario.sigalotti@inria.fr}

\begin{document}

\setstcolor{red}

\begin{abstract}
We prove the $C^{1}$ regularity for a class of abnormal length-minimizers in rank 2  sub-Riemannian structures.
As a consequence of our result, all length-minimizers for rank 2 sub-Riemannian structures of step up to 4 are of class $C^{1}$.
\end{abstract}

\maketitle
\setcounter{tocdepth}{1}
 \tableofcontents

\section{Introduction}

The question of regularity of length-minimizers is one of the main open problems in sub-Riemannian geometry, cf.\ for instance \cite[Problem 10.1]{montgomerybook} or \cite[Problem II]{AAA14} and the survey \cite{Monti14}.

Length-minimizers are solutions to a variational problem with constraints and satisfy a first-order necessary condition  resulting from  the Pontryagin Maximum Principle. With every length-minimizer $\gamma:[0,T]\to M$ we can associate a lift $\lambda:[0,T]\to T^{*}M$ in the cotangent space, satisfying a Hamiltonian equation. This lift can be either \emph{normal} or \emph{abnormal}, although  a length-minimizer $\gamma$ can actually admit several lifts,  each of them being either normal or  abnormal.

If a length-minimizer admits a normal lift, then it is smooth, i.e., $C^{\infty}$, since normal  lifts are solutions of smooth autonomous Hamiltonian systems in $T^{*}M$.
Note that we assume length-minimizers to be parametrized by arclength and  their regularity is meant with respect to this time parametrization.
The question of regularity is then reduced to length-minimizers that are strictly abnormal, i.e., those  which do not admit normal lifts. For such length-minimizers, from the first order necessary condition (and actually from the second order one as well) it is a priori not possible to deduce any regularity other than Lipschitz continuity.

In this paper we investigate the following.
\begin{conj}  \label{conjecture}
Are all
length-minimizers in a sub-Riemannian manifold  of class $C^{1}$?
\end{conj}

If the sub-Riemannian structure has step 2, there are no strictly abnormal length-minimizers, see e.g.\ \cite{Agrachev-Sarychev,ABB}, thus every length-minimizer admits a normal lift, and is  hence smooth. For step 3 structures, the situation is already more complicated and a positive answer to  the above problem is known only for Carnot groups (where, actually, length-minimizers are proved to be $C^{\infty}$), see \cite{LDLMV13,TY13}.  When the sub-Riemannian structure is analytic, more is known on the size of the set of points where a length-minimizer can lose regularity \cite{Sussmann2014ART}, regardless of the rank and of the step of the distribution.

\medskip
To state our main result, we introduce some notations. We refer the reader to Section~\ref{s:prel} for precise definitions.
Recall that a sub-Riemannian structure $(D,g)$ on $M$ is defined by a bracket generating distribution $\distr$ endowed with a metric $g$. Hence $\distr$ defines a flag of subspaces at every point $x \in M$
$$
\distr_x=\distr^{1}_x\subset \distr^{2}_x \subset \distr^{3}_x\subset \cdots \subset \distr^{r}_x=T_xM,
$$
where
$\distr^i_x$ is the subspace of the tangent space spanned by Lie brackets of length at most $i$ between horizontal vector fields. This induces a dual decreasing sequence of subspaces of $T^{*}_xM$
$$
0=(\distr^{r}_x)^{\perp}\subset \cdots \subset (\distr^{4}_x)^{\perp} \subset (\distr^{3}_x)^{\perp} \subset (\distr^{2}_x)^{\perp} \subset (\distr^{1}_x)^{\perp} \subset T^{*}_xM,
$$
where perpendicularity is considered with respect to the duality product. By construction, any abnormal lift satisfies  $\lambda(t)\in (\distr^{1})^{\perp}$ for every $t$. If the lift is strictly abnormal, then by Goh conditions $\lambda(t)\in (\distr^{2})^{\perp}$ for every $t$.

When the distribution has rank 2, it is known that if $\lambda(t)$ does not cross  $(\distr^{3})^{\perp}$, then the length-minimizer is $C^{\infty}$ \cite[Sect.~6.2, Cor.~4]{liu95}.
Our main result pushes this analysis  further and establishes that the answer to the
Open Problem
is positive
for length-minimizers whose abnormal lift does not enter $(\distr^{4})^{\perp}$.

\begin{theorem}\label{t:main}
Let $(\distr,g)$ be a rank 2 sub-Riemannian structure on $M$. Assume that $\gamma:[0,T]\to M$ is an abnormal minimizer parametrized by arclength. If $\gamma$ admits a lift
satisfying  $\lambda(t)\notin (\distr^{4})^{\perp}$ for every $t \in [0,T]$, then $\gamma$ is  of class $C^1$.
\end{theorem}

If the sub-Riemannian manifold has rank $2$ and step at most $4$, the assumption in Theorem \ref{t:main} is trivially satisfied by every abnormal minimizer $\gamma$ and we immediately obtain the following corollary.
\begin{corollary}\label{cor:step4}
Assume that the sub-Riemannian structure has rank $2$ and step at most $4$. Then all length-minimizers are of class $C^1$.
\end{corollary}

 It is legitimate to ask
whether the $C^1$ regularity in the Open Problem
can be further improved.
Indeed, the argument behind our proof permits to obtain $C^{\infty}$ regularity of length-minimizers under an additional nilpotency condition on the Lie algebra generated by horizontal vector fields.
\begin{proposition}\label{t:main2}
Assume that $\distr$ is generated by two vector fields $X_{1},X_{2}$ such that the Lie algebra $\mathrm{Lie}\{X_{1},X_{2}\}$ is  nilpotent of step at most $4$.
Then for every sub-Riemannian structure $(\distr,g)$ on $M$, the corresponding length-minimizers are of class $C^\infty$.
\end{proposition}
%
The above proposition applies in particular to Carnot groups of rank 2 and step at most 4. In this case we recover the results obtained in \cite[Example 4.6]{LM08}.

\medskip
The strategy of proof of Theorem \ref{t:main} is to show that, at points where they are not of class $C^1$, length-minimizers can admit only corner-like singularities. This is done by a careful asymptotic analysis of the differential equations satisfied by the abnormal lift, which exploits their Hamiltonian structure.
We can then conclude thanks to the following result.

\begin{theorem}[\cite{Hakavuori2016}]
\label{t:HLK}
Let $M$ be a sub-Riemannian manifold. Let $T>0$ and let $\gamma:[-T,T]\to M$ be a horizontal curve parametrized by arclength. Assume that, in local coordinates, there exist
$$\dot \gamma^{+}(0):=\lim_{{t\downarrow 0}} \frac{ \gamma(t)-\gamma(0)}{t},\qquad \dot \gamma^{-}(0):=\lim_{t\uparrow 0}\frac{ \gamma(t)-\gamma(0)}{t}.$$
If $\dot \gamma^{+}(0)\neq \dot \gamma^{-}(0)$, then $\gamma$ is not a length-minimizer.
\end{theorem}

We observe that the proof contained in  \cite{Hakavuori2016} requires a previous result stated in \cite{LM08}. A complete argument for the latter, addressing some issues raised in \cite[p.\  1113-15]{RiffordBBK}, is provided in  \cite{MPV17}.
For sub-Riemannian structures of rank $2$ and step at most $4$ (and indeed also for higher step, under an additional condition on the Lie algebra generated by horizontal vector fields), the fact that corners are not length-minimizers is already contained in \cite{LM08}.

We  notice that the answer to the Open Problem is known to be positive also in a class of rank 2 Carnot groups (with no restriction on the step, but satisfying other additional conditions). For this class of structures in \cite{Monti2014}, it is proved the $C^{1,\alpha}$ regularity for some suitable $\alpha>0$ (depending on the step).

We also refer to \cite{MPV18bis,HLD18} for recent results regarding these issues.

\subsection{Structure of the paper}
In Section \ref{s:prel} we recall some notations and preliminary notions. Section \ref{s:nilp} is devoted to a desingularization and nilpotentization argument.  Section \ref{s:dyn1} contains a preliminary analysis on the dynamics of abnormal extremals. To illustrate our approach in a simpler case, we discuss in Section~\ref{s:r2s4} the proof of the main result for a nilpotent structure of step up to $4$. Then in Sections~\ref{s:sis} and \ref{proof-ex-conj} we complete our analysis to prove the general result. Appendix~\ref{s:app} contains a technical lemma.

\subsection*{Acknowledgments}
We thank Ugo Boscain, Paolo Mason, Ludovic Rifford, and Luca Rizzi for many stimulating discussions.
This work was supported by the Grant ANR-15-CE40-0018 SRGI ``Sub-Riemannian geometry and interactions''
and by a public grant as part of the Investissement d'avenir project, reference ANR-11-LABX-0056-LMH, LabEx LMH, in a joint call with Programme Gaspard Monge en Optimisation et Recherche Op\'erationnelle.

\section{Notations and preliminary notions}
\label{s:prel}

Let $M$ be a smooth  $n$-dimensional manifold.  A \emph{sub-Riemannian structure} of rank $m$ on $M$ is a triplet $(E,g^E,f)$ where $E$ is a vector bundle of rank $m$ over $M$,  $g^E$ is an Euclidean metric on $E$, and $f: E \to TM$ is a morphism of vector bundles
such that $f(E_x)\subseteq T_x M$ for every $x\in M$.
Fix such a structure and define a family of subspaces of the tangent spaces by
\begin{equation*}
\distr_x=\{X(x) \mid X \in\distr\} \subseteq T_{x}M,\qquad \forall\, x\in M,
\end{equation*}
where $\distr=\{f \circ Y \mid Y \textrm{ smooth section of }E\}$ is a submodule of the set of vector fields on $M$. We assume that the structure is \emph{bracket generating}, i.e., the tangent space $T_{x}M$ is spanned by the vector fields in $\distr$ and their iterated Lie brackets evaluated at $x$.

The sub-Riemannian structure induces a quadratic form $g_{x}$ on $\distr_{x}$ by
\begin{equation*}
g_x(v,v)=
\inf\{ g^E_x( u,u) \mid f(u)=v, u\in E_x\}, \quad v\in \distr_x.
\end{equation*}
In analogy with the classic sub-Riemannian case and to simplify notations, in the sequel we will refer to the sub-Riemannian structure as the pair $(\distr,g)$ rather than $(E, g^E, f)$. This is justified since all the constructions and definitions below rely only on $\distr$ and $g$. The  triplet $(M, \distr, g)$ is called a \emph{sub-Riemannian manifold}.

\begin{rmk}
Usually, a sub-Riemannian manifold  denotes a triplet $(M,D,g)$, where $M$ is a smooth manifold, $D$
is a subbundle of $TM$, and $g$ is a Riemannian metric on $D$  (see, e.g., \cite{Bellaiche1996}).
This corresponds to the case where $f(E_x)$ is of constant rank.
The definition given above follows, for instance, \cite{ABB}.
\end{rmk}

A \emph{horizontal curve} $\gamma : [0,T] \to M$ is an absolutely continuous path such that $\dot\gamma(t)\in \distr_{\gamma(t)}$ for almost every (a.e.\ for short) $t\in [0,T]$.
The \emph{length} of a horizontal curve is defined by
\begin{equation*}
\ell(\gamma) = \int_0^T \sqrt{g_{\gamma(t)}(\dot\gamma(t),\dot\gamma(t))}dt.
\end{equation*}
The \emph{sub-Rieman\-nian distance} between two arbitrary points $x,y$ in $M$ is then
\begin{equation*}
d(x,y) = \inf\{\ell(\gamma)\mid \gamma(0) = x,\, \gamma(T) = y,\, \gamma \text{ horizontal} \}.
\end{equation*}
A \emph{length-minimizer} is a horizontal curve $\gamma$ which realizes the distance between its extremities, that is, $\ell(\gamma)=d(\gamma(0),\gamma(T))$. Note that any time-reparametrization of a length-minimizer is a length-minimizer as well.

A \emph{generating frame} of the sub-Riemannian structure is a family of smooth vector fields $X_1,\dots,X_k$ such that $\distr$ is generated by $X_1,\dots, X_k$ as a module and
$$
g_x(v,v)=\inf\left\{\sum_{i=1}^{k}u_{i}^{2}\mid \sum_{i=1}^{k}u_{i}X_{i}(x)=v\right\}, \quad x \in U, \ v \in \distr_x.
$$
There always exists a global generating frame (see \cite[Corollary 3.26]{ABB}), with, in general, a number $k$ of elements greater than the rank $m$ of the structure. However, every point $x \in M$ admits a neighborhood on which there exists a (local) generating frame with exactly $k=m$ elements, e.g., by taking the image via $f$ of a local orthonormal frame of $(E,g^E)$.

Fix now a (local or global) generating frame $X_1,\dots,X_k$ of $(D,g)$. For any horizontal curve $\gamma$ of finite length,  there exists $u\in L^{\infty}([0,T],\R^{k})$ satisfying
\begin{equation}\label{eq:dotgamma}
\dot\gamma(t) =  \sum_{i=1}^k u_i(t) X_i(\gamma(t)), \qquad \text{for a.e.}\ t \in [0,T].
\end{equation}
The curve is said to be \emph{parametrized by arclength} if $g_{\gamma(t)}(\dot\gamma(t),\dot\gamma(t)) =1$ for a.e.\ $t \in [0,T]$, i.e., if there exists $u\in L^\infty([0,T],\mathbb{S}^{k-1})$ satisfying~\eqref{eq:dotgamma}. In that case $\ell(\gamma)=T$.

\smallskip
To state the first order necessary conditions, let us first introduce some notations. For $\lambda \in T^*M$ and $x=\pi(\lambda)$, where $\pi : T^*M \to M$ is the canonical projection, we set $h_i(\lambda) = \langle\lambda, X_i(x)\rangle$, for $i=1,\dots,k$ (here $\langle \lambda,\cdot\rangle$ denotes the dual action of covectors on vectors). Recall also that, for a function $H : T^*M \to \R$, the corresponding Hamiltonian vector field $\vec{H}$ is the unique vector field such that $\sigma(\cdot,\vec{H}) = dH$, where $\sigma$ is the canonical symplectic form on the cotangent bundle.

Applying the Pontryagin Maximum Principle to the sub-Riemannian length minimization problem yields the following theorem.
\begin{theorem} \label{t:utile}
Let $(M,D,g)$ be a sub-Riemannian manifold with generating frame $X_1,\dots, X_k$ and $\gamma  :[0,T] \to M$
be a length-minimizer. Then there exists a nontrivial absolutely continuous  curve $t\mapsto \lambda(t) \in T_{\gamma(t)}^*M$  such that one of the following conditions is satisfied:
\begin{itemize}
\item[(N)] $\dot{\lambda}(t) = \vec{H}(\lambda(t))$  for all $t\in [0,T]$, where $H(\lambda) = \frac{1}{2}\sum_{i=1}^{k}h_i^2$,
\item[(A)] $\dot{\lambda}(t) =  \sum_{i=1}^{k} u_i (t) \vec{h}_i(\lambda(t))$
 for almost every $t\in [0,T]$, with $u_1,\dots,u_k\in L^1([0,T])$. Moreover,
 $\lambda(t) \in (D_{\gamma(t)})^\perp$ for all $t$, 
i.e.,  $h_i(\lambda(t)) \equiv 0$ for $i=1,\dots,k$.
\end{itemize}
\end{theorem}

In  case (N) (respectively, case (A)), $\lambda$ is called a \emph{normal} (respectively, \emph{abnormal}) \emph{extremal}. Normal extremals are integral curves of $\vec{H}$. As such, they are smooth. A length-minimizer is normal (respectively, abnormal) if it admits a normal (respectively, abnormal) extremal lift. We stress that both conditions can be satisfied for the same curve $\gamma$, with different lifts $\lambda_{1}$ and $\lambda_{2}$.

\section{Desingularisation and nilpotentization}\label{s:nilp}

\subsection{Desingularisation}
Let $(M, \distr, g)$ be a sub-Riemannian manifold.
We define recursively  the following
 sequence of submodules of the set of vector fields,
$$
\distr^{1}=\distr,\quad \distr^{i+1}=\distr^{i}+[\distr,\distr^{i}].
$$
At every point $x \in M$, the evaluation at $x$ of these modules induces a flag of subspaces of the tangent space,
\begin{equation*}
\distr^1_x \subset \distr^2_x \subset \cdots \subset \distr^r_x = T_xM.
\end{equation*}
The smallest integer $r=r(x)$ satisfying $\distr^r_x = T_xM$ is called the \emph{step of $\distr$ at $x$}. A point is said to be \emph{regular} if the dimensions of the subspaces of the flag are locally constant  in an open neighborhood of the point. When every point in $M$ is regular, the sub-Riemannian manifold is said to be \emph{equiregular}.

In general a sub-Riemannian manifold may admit non-regular points. However, for our purposes, we can restrict ourselves with no loss of generality to  equiregular manifolds thanks to a desingularisation procedure.

\begin{lemma}
\label{le:desing}
Fix an integer $m\ge 2$. Assume that for every rank $m$ equiregular sub-Riemannian structure the following property holds: every arclength parametrized abnormal minimizer
admitting a lift  $\lambda(t)\notin (\distr^{4})^{\perp}$
is  of class $C^1$. Then the same property holds true  for every rank $m$ sub-Riemannian structure.
\end{lemma}

\begin{proof}
Let $(M, \distr, g)$ be a non-equiregular sub-Riemannian manifold of rank $m$ and $\gamma$ be an abnormal length-minimizer of $(M, \distr, g)$ which admits an abnormal extremal lift such that $\lambda(t)\notin (\distr^{4})^{\perp}$ for every $t \in [0,T]$. Assume moreover that $\gamma$ is parametrized by arclength. We have to prove that $\gamma$ is  of class $C^1$.

Fix  $t_0 \in [0,T]$ and a generating frame  $X_1,\dots,X_m$ on a neighborhood of $\gamma(t_0)$.
By \cite[Lemma 2.5]{Jea-2014}, there exists an equiregular sub-Riemannian manifold $(\widetilde{M},\widetilde{\distr},\widetilde{g})$ of rank $m$ with a generating
frame $\xi_1,\dots,\xi_m$ and a map $\varpi: \widetilde{M} \to M$  onto a neighborhood $U \subset M$ of $\gamma(t_0)$ such that $\varpi_*\xi_i=X_i$. Up to reducing the interval $[0,T]$ we assume that $\gamma(t) \in U$ for all $t\in [0,T]$. Let $u\in L^\infty([0,T],\mathbb{S}^{m-1})$ be  such that
\begin{equation*}
\dot\gamma(t) =  \sum_{i=1}^m u_i(t) X_i(\gamma(t)), \qquad \mathrm{a.e.}\ t.
\end{equation*}
By construction, since $\gamma$ is a length-minimizer, there exists a length-minimizer $\widetilde{\gamma}$ in $\widetilde{M}$ with $\varpi(\widetilde{\gamma})=\gamma$ associated with the same $u$, that is,
\[
\dot{\widetilde{\gamma}}(t) = \sum_{i=1}^m u_i(t)\xi_i(\widetilde{\gamma}(t)), \qquad \mathrm{a.e.}\ t,
\]
which is parametrized by arclength as well. Hence the trajectory $\gamma$
has at least the same regularity as $\widetilde{\gamma}$.

Moreover, if $\lambda$ is an abnormal lift of $\gamma$ in $T^*M$, then $\widetilde{\gamma}$ admits an abnormal lift $\widetilde{\lambda}$ in $T^*\widetilde{M}$ such that $\widetilde{\lambda}(t)=\varpi^* \lambda(t)$ for every $t$. Since  $\varpi^* (\widetilde{\distr}^k)^\perp = ({\distr}^k)^\perp$ for any positive integer $k$, the property $\lambda(t) \notin ({\distr}^4)^\perp$ implies $\widetilde{\lambda}(t) \notin (\widetilde{\distr}^4)^\perp$.

It results from the hypothesis that $\widetilde{\gamma}$ is $C^1$,  so $\gamma$ is of class $C^1$ in an open neighborhood of $t_0 \in [0,T]$, which ends the proof.
\end{proof}

As a consequence of Lemma~\ref{le:desing}, we can assume in the rest of the paper that the sub-Riemannian manifold is equiregular.

\subsection{Nilpotentization}
Let us recall the construction of the nilpotent approximation (see for instance \cite{Bellaiche1996} for details).

Let $(M, \distr, g)$ be an equiregular sub-Riemannian manifold. We fix a point $x \in M$ and a local generating frame $X_1,\dots,X_m$ in a neighborhood of $x$.

For $i=1,\dots,n$, let $w_i$ be the smallest integer $j$ such that $\dim \distr^j_x \geq i$. We define the dilations $\delta_\nu : \mathbb{R}^n \to \mathbb{R}^n$ for $\nu \in \R$ as $\delta_\nu (z) = (\nu^{w_1} z_1, \dots, \nu^{w_n} z_n)$.  Let
$z^x$ be a system of privileged coordinates at $x$   and set $\delta_\nu^x=\delta_\nu \circ z^x$. Then, for $i=1,\dots,m$, the 	
 vector field $\varepsilon \big(\delta^{x}_{1/\varepsilon}\big)_*X_i$ converges locally uniformly as $\varepsilon \to 0$ to a vector field $\widehat X^x_i$ on $\R^n$. The space $\R^n$ endowed with the sub-Riemannian structure having $\widehat X^x_1, \dots, \widehat X^x_m$ as generating frame is called the \emph{nilpotent approximation of $(M, \distr, g)$ at $x$} and is denoted by $\widehat{M}_{x}$. This nilpotent approximation $\widehat{M}_{x}$ is a Carnot group equipped with a left-invariant sub-Riemannian structure.

Since $(M,D,g)$ is equiregular, we can locally choose systems of privileged coordinates $z^x$ depending continuously on $x$ \cite[Sect.\ 2.2.2]{Jea-2014}. Note that the $w_i$'s and $\delta_\nu$ are independent of  $x$. Thus an easy adaptation of the proof of \cite[Prop.\ 3.4]{Ambrosio2015} (see also \cite[Sect.\ 10.4.1]{ABB}) shows that, for $i=1,\dots,m$, the  vector field
$\varepsilon \big( \delta^{x}_{1/ \varepsilon} \big)_*X_i$ converges locally uniformly to $\widehat X_i^{x_0}$ as $\varepsilon \to 0$ and $x \to x_0$.

\begin{lemma}\label{lem:nilpot}
	Let $(a_n)_{n\in\mathbb N},(b_n)_{n\in\mathbb N}\subset [0,T]$, $\bar a\in [0,T]$, be such that $a_n,b_n\rightarrow \bar a$ and $a_n<b_n$ for any $n\in\mathbb N$.
	Given $u\in L^\infty([0,T],\mathbb{S}^{m-1})$  and $n\in\N$, define $u_n\in L^\infty([0,1],\mathbb{S}^{m-1})$ by
	\[ u_n(\tau)=u(a_n+\tau(b_n-a_n)).\]
	Assume that the sequence
	$(u_n)_{n\in\mathbb N}$  converges to
 $u_\star\in L^\infty([0,1],\mathbb R^{m})$ for the weak-$\star$ topology of $L^\infty([0,1],\mathbb{R}^{m})$  and, moreover, that the trajectory $\gamma:[0,T]\to M$ associated with $u$  is a length-minimizer. If $x=\gamma(\bar a)$, then the trajectory $\gamma_\star:[0,1]\to \widehat M_{x}$ satisfying
$$
\dot\gamma_\star(s) = \sum_{i=1}^m u_{\star,i}(s)\widehat X_i^x(\gamma_\star(s)),\qquad \gamma_\star(0)= 0,
$$
is also a length-minimizer. In particular, $u_\star(t)\in\mathbb S^{m-1}$ for almost every $t\in [0,1]$.
\end{lemma}

\begin{proof}
We consider a continuously varying family of privileged coordinates $z^{\gamma(t)}$, $t\in[0,T]$, and the corresponding 1-parameter family of dilations $\delta_\nu^t:=\delta_\nu^{\gamma(t)}$.
It is not restrictive to assume
 that $\delta^{a_n}_{\frac1{b_n-a_n}}\gamma(t)$ is well-defined  for every $n\in\mathbb N$ and $t\in[a_n,b_n]$.

	Let $\gamma_n$ be defined by $\gamma_n(\tau)=\delta^{a_n}_{\frac1{b_n-a_n}}\left(\gamma(a_n+\tau(b_n-a_n))\right)$. Then, $\gamma_n$ is a length-minimizing curve for the sub-Riemannian structure on $\R^n$ with orthonormal frame
	\[(b_n-a_n)\Big(\delta^{a_{n}}_{\frac{1}{b_n-a_n}}\Big)_*X_1, \ldots, (b_n-a_n)\Big(\delta^{a_{n}}_{\frac{1}{b_n-a_n}}\Big)_*X_m.\]
	 The corresponding control is $u_n$.

	Since the sequence $\Big((b_n-a_n)\Big(\delta^{a_{n}}_{\frac{1}{b_n-a_n}}\Big)_*X_i\Big)_{n\in\mathbb N}$ converges locally uniformly to $\widehat X_i^x$, it follows by standard ODE theory that $(\gamma_n)_{n\in\mathbb N}$  converges uniformly to $\gamma_\star$.

We claim that $\widehat d(\gamma_\star(0),\gamma_\star(1))=1$. Indeed,
$\ell(\gamma_\star)\ge \widehat d(\gamma_\star(0),\gamma_\star(1))$ and, by \cite[Theorem 7.32]{Bellaiche1996}, we have
\begin{align*}
		\widehat d(\gamma_\star(0),\gamma_\star(1)) &= \lim_{n\to\infty} \frac{1}{b_n-a_n} d(\gamma(a_n), \gamma (b_n))
\\ &= \lim_{n\to\infty} \int_0^1 |u(a_n+\tau(b_n-a_n))|\,d\tau = 1,
	\end{align*}
where $|\cdot|$ denotes the norm in $\R^{m}$.
	 On the other hand, by weak-$\star$ convergence we have
	\begin{equation*}
		\ell(\gamma_\star)=\|u_\star\|_{L^1([0,1],\mathbb R^{m})}\le \liminf_{n\to\infty}\|u_n\|_{L^1([0,1],\mathbb R^{m})}=1,
	\end{equation*}	
	 proving the claim.
	
	To conclude the proof, it suffices now to observe that the above implies that $\gamma_\star$ is minimizing. In particular, since $|u_\star(t)|\leq 1$ a.e.\ on $[0,1]$ by the properties of weak-$\star$ convergence, this shows that $|u_\star(t)|= 1$ a.e.\ on $[0,1]$.
\end{proof}

\begin{cor}
Let $\gamma$, $u$, $\bar a$, $(a_n)_{n\in\mathbb N}$, $(b_n)_{n\in\mathbb N}$, and $u_\star$ be as in Lemma~\ref{lem:nilpot}. Assume that there exist $u_+,u_-\in \mathbb{S}^{m-1}$ such that $u_\star=u_-$ almost everywhere on $[0,1/2]$ and
$u_\star=u_+$ almost everywhere on $[1/2,1]$.
Then $u_-=u_+$.
\end{cor}
\begin{proof}
If $u_-\ne u_+$, then $\gamma_\star$ is not length-minimizing by Theorem~\ref{t:HLK}, which contradicts Lemma~\ref{lem:nilpot}.
\end{proof}

\section{Dynamics of abnormal extremals: preliminary results} \label{s:dyn1}

In this section we present the dynamical system associated with the abnormal extremal, whose analysis is the basis for the proof of Theorem~\ref{t:main}, and we derive a first result on its structure.

\subsection{Introduction to the dynamical system} \label{s:intro_dyn}
Let $(M, \distr, g)$ be an equiregular sub-Riemannian manifold of rank $2$. Since the arguments are local, in what follows we fix  a local generating frame $\{X_1,X_2\}$ of $(\distr,g)$.

Consider an abnormal length-minimizer $\gamma:[0,T]\to M$ parametrized by arclength. Then $T=d(\gamma(0),\gamma(T))$ and there exists $u\in L^\infty([0,T],\mathbb{S}^1)$ such that
\begin{equation*}
	\dot\gamma(t) = u_1(t)X_1(\gamma(t)) + u_2(t)X_2(\gamma(t)),\quad \mbox{a.e.~}t\in[0,T].
\end{equation*}
Moreover from Theorem~\ref{t:utile}, $\gamma$ admits a lift $\lambda:[0,T]\to T^*M$ which satisfies
$$
\dot{\lambda}(t) =   u_1 \vec{h}_1(\lambda(t)) + u_2 \vec{h}_2(\lambda(t)) \quad \hbox{and} \quad  h_1(\lambda(t)) \equiv h_2(\lambda(t)) \equiv 0.
$$
 By a slight abuse of notation, set $h_i(t) = \langle\lambda(t), X_i(\gamma(t))\rangle$, $i=1,2$, and
for every $i_1,\dots, i_m\in\{1,2\}$,
\begin{equation*}
	h_{i_1\cdots i_m}(t) =
	\langle \lambda(t), [X_{i_1},\ldots,[X_{i_m-1},X_{i_m}]](\gamma(t)) \rangle.
\end{equation*}
Such a function
$h_{i_1\cdots i_m}$ is  absolutely continuous and satisfies
\begin{equation}\label{eq:poisson}
	\dot h_{i_1\cdots i_m}(t) = u_1(t)h_{1i_1\cdots i_m}(t) + u_2(t)h_{2i_1\cdots i_m}(t) \quad \text{for a.e. } t\in[0,T].
\end{equation}
Differentiating the equalities $h_1 \equiv h_2 \equiv 0$ and using \eqref{eq:poisson} we obtain $h_{12} \equiv 0$. Differentiating again we get
\begin{equation}\label{eq:diff-h}
	0 = \dot h_{12} = u_1 h_{112}+ u_2 h_{212} \quad \text{a.e. on } [0,T].
\end{equation}

\begin{rmk}
The identities $h_1(t)=h_2(t)=h_{12}(t)=0$ imply that $\lambda(t)\in (\distr^{2})^{\perp}$ for every $t$. The latter  is known as \emph{Goh condition} and is in general (i.e., for sub-Riemannian structures of any rank) a necessary condition for the associated curve to be length-minimizing \cite{AS99}.
 It is known that a generic sub-Riemannian structure of rank larger than $2$  does not have non-constant abnormal extremals satisfying the Goh condition \cite{CJT}.
\end{rmk}

Let $h=(-h_{212},h_{112})$ and $(t_0,t_1)\subset(0,T)$ be a maximal (i.e., non-extendable)
open interval on which $h\neq0$.  Equation~\eqref{eq:diff-h} then implies that $u =\pm h/|h|$ almost everywhere on $(t_0,t_1)$.

Moreover, by length-minimality  of $\gamma$  we can assume without loss of generality
that $u = h/|h|$  on $(t_0,t_1)$ (see Lemma~\ref{lem:constantsign} in the appendix).
Thus $\gamma$ may be non-differentiable only at a time $t$ such that $h(t)=0$. In particular, if the step of the sub-Riemannian structure is not greater than $3$, then $\gamma$ is differentiable everywhere. We assume from now on that the step is at least $4$.

Observe that
from \eqref{eq:poisson} and using $u=h/|h|$ one obtains
\begin{equation}\label{eq:Ah}
	\dot h = A\frac{h}{|h|},
	\qquad
	A = \left(
	\begin{array}{cc}
		-h_{2112} & -h_{2212} \\
		h_{1112} & h_{2112} \\
	\end{array}
	\right), \qquad \text{on }(t_0,t_1).
\end{equation}
Here, we used the  relation $h_{1212} = h_{2112}$, which follows from the Jacobi identity
\begin{align*}
[X_{1},[X_{2},[X_{1},X_{2}]]]&=-[[X_{1},X_{2}],[X_{1},X_{2}]]-[X_{2},[[X_{1},X_{2}],X_{1}]]
=[X_{2},[X_{1},[X_{1},X_{2}]]].
\end{align*}
Observe that the matrix $A$ has zero trace and is  absolutely continuous on the whole interval $[0,T]$.

\begin{lemma}\label{l:dis}
Assume that
$\lambda(t) \notin (\distr^{4}_{\gamma(t)})^{\perp}$ for every $t\in[0,T]$. If $h(t_{0})=0$ for some $ t_{0}\in [0,T]$, then
$A(t_{0})\ne 0$.
\end{lemma}

\begin{proof}
The fact that $\gamma$ is abnormal implies that the non-zero covector $\lambda(t)$ annihilates $\distr_{\gamma(t)}$ for every $t\in[0,T]$. The Goh condition $h_{12} \equiv 0$ guarantees that it also annihilates $\distr^{2}_{\gamma(t)}$. The fact that $h(t_{0})=0$ says, moreover, that $\lambda(t_{0})$ annihilates  $\distr^{3}_{\gamma(t_0)}$.
If $A(t_{0})$ is equal to zero, then $\lambda(t_{0})$ annihilates  $\distr^{4}_{\gamma(t_0)}$, which contradicts the assumption.
\end{proof}

\subsection{The sign of $\det A$ is non-negative where $h$ vanishes}

A key step in the proof of Theorem~\ref{t:main} is the following result.

\begin{prop}\label{prop:nonpos}
Let $(t_0,t_1)$ be a maximal open interval of $[0,T]$ on which $h\ne 0$ and assume that $t_1<T$. Then
$\det A(t_1)\le 0$.
\end{prop}

\begin{proof}
	Assume by contradiction that $\det A(t_1)>0$. Since $\tr A(t_{1}) = 0$, there exists $P\in {\rm GL}(2,\R)$ such that
	\begin{equation}\label{eq:P}
		P^{-1}A(t_1)P =
		\left(
		\begin{array}{cc}
			0 & -a \\
			a & 0
		\end{array}
		\right), \qquad  a>0.
	\end{equation}
	Define  the scalar functions $\alpha$, $\beta$ and $\zeta$ through the relation
	\begin{equation*}
		 P^{-1}A(t)P = \left(
		\begin{array}{cc}
			-\alpha(t) & \beta(t) \\
			\zeta(t) & \alpha(t)
		\end{array}
		\right),
	\end{equation*}
and notice that  $\alpha,\beta,\zeta$ are
absolutely continuous with bounded derivatives on $(t_0,t_1)$, since they are
linear combinations of $h_{2112}, h_{2212}, h_{1112}$, according to \eqref{eq:poisson}. Clearly, \eqref{eq:P} implies that $\alpha(t)\rightarrow 0$, $\beta(t)\rightarrow- a$, and $\zeta(t)\rightarrow a$ as $t\to t_1$.

	Consider a time rescaling and  a polar coordinates representation so that
	$P^{-1}h(t) = \rho(s(t)) e^{i\theta(s(t))}$, where $$s(t):=\int_{t_{0}}^{t}\frac{|P^{-1}h(\tau)|}{|h(\tau)|}d\tau.$$
	It is useful to introduce $\mu := (\zeta+\beta)/2 $ and $\eta := (\zeta-\beta)/2$.
	Then, denoting by $\rho'$ and $\theta'$ the derivatives of $\rho$ and $\theta$ with respect to the parameter $s$, \eqref{eq:Ah} can be rewritten as
	\begin{equation*}
	\begin{cases}
		\rho'= (-\alpha\cos2\theta + \mu\sin2\theta), \\
		\theta'  =\frac{1}{\rho}(\alpha\sin2\theta+\mu\cos2\theta+\eta).
		 \end{cases}
	\end{equation*}
Let $w=\alpha\sin2\theta+\mu\cos2\theta+\eta$ and  	notice that $2a>w>a/2$ in a left-neighborhood of $s(t_{1})$.
Therefore,
\begin{align*}
(\rho^2w)' &= 2\rho(-\alpha\cos2\theta + \mu\sin2\theta) w+\rho^2(\alpha'\sin2\theta+\mu'\cos2\theta+\eta')\\
&\quad +\rho^2 (\alpha\cos2\theta-\mu\sin2\theta)2\theta'\\
 &=\rho^2w\frac{\alpha'\sin2\theta+\mu'\cos2\theta+\eta'}{w}\ge -M \rho^2w,
\end{align*}
for some constant $M>0$.  This implies at once that $t\mapsto e^{Mt}\rho^2(t)w(t)$ is increasing, and hence
that it is impossible for $\rho^2w$ to tend to zero as $s\rightarrow s(t_1)$. This contradicts the assumption that $\rho(t)\to 0$ as $t\rightarrow t_{1}$, completing the proof of the statement.
\end{proof}

\section{Dynamics of abnormal extremals in a special case: proof of Proposition~\ref{t:main2}} \label{s:r2s4}
In this section we prove Proposition~\ref{t:main2}.
We present it here to illustrate in a simpler context the general procedure
used later to complete the proof of Theorem~\ref{t:main}.

Assume that $\distr$ is generated by two vector fields $X_{1},X_{2}$ such that the Lie algebra $\mathrm{Lie}\{X_{1},X_{2}\}$ is nilpotent  of step at most $4$. This means that all Lie brackets of $X_{1},X_{2}$ of length 5 vanish.
%
In particular $\dim M\leq 8$.

\begin{proof}[Proof of Proposition \ref{t:main2}] Without loss of generality, we assume that the step is equal to 4. Recall that for an abnormal minimizer  on an interval $I$ we have
\begin{equation*}
h_1\equiv h_2\equiv h_{12}\equiv 0, \qquad 0 = \dot h_{12} = u_1 h_{112}+ u_2 h_{212} \quad \text{a.e. on } I.
\end{equation*}
The vector $h=(-h_{212},h_{112})$ satisfies the differential equation
\begin{equation}\label{eq:mastereq}
	\dot h = Au,
	\qquad
	A = \left(
	\begin{array}{cc}
		-h_{2112} & -h_{2212} \\
		h_{1112} & h_{2112} \\
	\end{array}
	\right), \quad \text{a.e. on } I.
\end{equation}
Notice that $A$ is a constant matrix (with zero trace), as follows from \eqref{eq:poisson} and the nilpotency assumption.

As we have already seen in the general case, on every interval
where $h(t)\neq 0$ we have
that $u$ is smooth and equal to  either
$\frac{h(t)}{|h(t)|}$ or $-\frac{h(t)}{|h(t)|}$.

 We are then reduced to the case where $h$ vanishes at some point $\bar t \in I$. In this case the matrix $A$ cannot be zero,
 as it follows from Lemma~\ref{l:dis}.

We consider  the following alternative:
\begin{itemize}
\item[(a)] $h(t)=0$ for all $t\in I$;
\item[(b)] $h$ does not vanishes identically on $I$.
\end{itemize}
Case (a). From \eqref{eq:mastereq} it follows that $u(t)$ is in the kernel of $A$ for a.e. $t\in I$. Since $u$ is nonzero for a.e. $t\in I$, then necessarily $A$ has one-dimensional kernel $\ker A=\mathrm{span}\{\bar u\}$, where $\bar u$ has norm one.
Then $u(t)=\sigma(t)\bar u$  for a.e.\ $t\in I$, with $\sigma(t)\in\{-1,1\}$  and
\begin{gather*}
	\dot\gamma(t) = \sigma(t) X_{\bar u}(\gamma(t)),\quad {\mbox{a.e.~$t\in I$,}}
\end{gather*}
with $X_{\bar u}$ a constant vector field. Since $\gamma$ is a length-minimizer then $\sigma$ is constant, and $u$ is smooth, thanks to Lemma~\ref{lem:constantsign} in the appendix.

\noindent Case (b).
Consider a maximal interval $J=(t_{0},t_{1})$ on which $h$ is never vanishing. Since $J\subsetneq I$, then either $h(t_{0})=0$ or $h(t_{1})=0$.

	The trajectories of \eqref{eq:Ah} are time reparametrizations of those of the linear system $\dot z =Az$. Hence $h$ stays in the stable or in the unstable manifold of $A$.
Recall that $\det A\le 0$ by Proposition~\ref{prop:nonpos} and notice that if $\det A= 0$ then $A$ is conjugate to the nilpotent matrix
$\begin{psmallmatrix} 0& 1\\0& 0
\end{psmallmatrix}
$.
Hence stable and unstable manifolds reduce to zero. We  deduce that  $\det A<0$.

Denote by $\lambda_{\pm}$ the eigenvalues of $A$ and by $v_{\pm}$ the corresponding  unit eigenvectors. Since $ h$ belongs to the stable (respectively, unstable) manifold of $A$ then  $\frac{h(t)}{|h(t)|}$ is constantly equal to $v_{-}$ or $-v_{-}$ on $J$ (respectively, $v_{+}$ or $-v_{+}$).
	Fix $t_{*}\in J$. Then  integrating \eqref{eq:mastereq} we get
	$$h(t)=h(t_*) \pm (t-t_*)\lambda_{-}v_{-},\qquad t\in J,$$
	or
	$$h(t)=h(t_*) \pm (t-t_*)\lambda_{+}v_{+},\qquad t\in J.$$
If $h(t_{1})=0$, then $\lim_{t\downarrow t_{0}}h(t)\neq 0$ and $J=I\cap (-\infty,t_{1})$.
Similarly, if $h(t_{0})=0$ then
 $\lim_{t\uparrow t_{1}}h(t)\neq 0$ and $J=I\cap (t_{0},+\infty)$.

If there exist two distinct maximal intervals of $I$ on which $h$ is never vanishing, then necessarily  there exist $\tau_{1}\leq \tau_{2}$  in $I$ such that these maximal intervals are of the form $J_{1}=I\cap (-\infty,\tau_{1})$ and $J_{2}=I\cap (\tau_{2},+\infty)$.  Notice that $h$ vanishes on $[\tau_{1},\tau_{2}]$.

If $\tau_{1}<\tau_{2}$, we can apply case (a) on the interval $(\tau_{1},\tau_{2})$, which leads to a contradiction since $A$ should have nontrivial kernel.
We are thus left to consider the case where $\tau_1=\tau_2=\bar t$, that is, when $h(t)\neq 0$ for $t\in I\setminus \{\bar t \}$. In this case  $u$ is  piecewise constant on $I\setminus \{\bar t \}$ and satisfies
\[
\lim_{t\downarrow \bar t} u(t)\in\{v_{-},-v_{-},v_{+},-v_{+}\},\quad \lim_{t\uparrow \bar t} u(t)\in\{v_{-},-v_{-},v_{+},-v_{+}\}.
\]
Theorem~\ref{t:HLK} and the length-minimizing assumption on $\gamma$ imply that
the two limits must be equal.
Hence, $u$ is constant on $I$, and in particular it is smooth.
\end{proof}

\begin{rmk}
  The technical ingredients of the above proof open the way to an alternative approach to the Sard conjecture for minimizers  \cite{AAA14}
   which is known in the free case \cite{LDMOPV16}.
   Indeed,
  assume that the hypotheses of Proposition \ref{t:main2}
  hold true and fix a  point $x\in M$.  We have proved that given any initial covector in $(\distr^2_x)^\perp$ there exist at most four length minimizing curves whose extremal lift starts with this covector.
  Hence, such curves are parametrized by at most $n-3$ parameters.
  By taking into account the time parametrization, the set of final points of abnormal minimizers starting from $x$ has codimension at most $2$.

   For recent results on the Sard conjecture for rank 2 structures in 3-dimensional manifolds, see \cite{BelottoRifford}, which extends the analysis in \cite{ZZ95}.
\end{rmk}

\section{Dynamics of abnormal extremals: the general case}\label{s:sis}

The goal of this section is to prove the following result.

\begin{prop}\label{prop:h/h}
Let $(t_0,t_1)$ be a maximal interval on which $h\ne0$. Assume that $t_1<T$ and $A(t_1)\ne0$. Then $u(t)$ has a limit as $t\uparrow t_1$, which is an eigenvector of $A(t_1)$. 
\end{prop}

We split the analysis in two steps. The first one, which is a rather straightforward adaptation of the  proof of Proposition~\ref{t:main2}, corresponds to the case  where $\det A(t_1)<0$.
We will then turn to the case where $\det A(t_1)=0$ (recall that, according to
Proposition~\ref{prop:nonpos}, $\det A(t_1)$ cannot be positive).

For this purpose, we start by proving a preliminary result.
\subsection{A time-rescaling lemma}
The result below highlights the fact that equation \eqref{eq:Ah} is ``almost invariant'' with respect to similarity of $A$.

\begin{lemma}\label{lem:s}
	For $P\in {\rm GL}(2,\R)$ and $t_*\in (t_0,t_1)$, we
consider the time reparameteri\-za\-tion given by
	\begin{equation*}
		\varphi:  [t_*,t_1)\ni t\mapsto s := \int_{t_*}^t \frac{d\tau}{|h(\tau)|}.
	\end{equation*}
Let $\mathfrak h=P^{-1}h\circ \varphi^{-1}$ and $\mathfrak A=P^{-1 }(A\circ \varphi^{-1})P$.
	Then,
	\begin{enumerate}[label=\rm(\roman*)]
		\item
		$\varphi(t)\rightarrow +\infty$ as $t\rightarrow t_1$;
		\item
		for any $p\in [1,+\infty]$ we have $\mathfrak h\in L^p((0,+\infty), \R^2)$;
		\item
		\label{point:A'h}
		for every $s\in (0,+\infty)$
		we have
		\begin{equation}\label{eq:A'h}
			{\mathfrak h}'(s) = \mathfrak A(s)\mathfrak h(s).
		\end{equation}
	\end{enumerate}
\end{lemma}

\begin{proof}
	We start by proving point \ref{point:A'h}. Observe that $\dot \varphi = 1/|h|$. Then, simple computations yield
	\begin{equation*}
		{\mathfrak h}' = \frac{P^{-1} \dot h\circ \varphi^{-1}}{\dot \varphi \circ \varphi^{-1}} = \mathfrak A\,\mathfrak h.
	\end{equation*}
	
	Assume now that $\lim_{t\rightarrow t_1} \varphi(t) = s_*<+\infty$. Then, since $ \mathfrak h(s_*) = h(t_1) = 0$, we have that $\mathfrak h$ is the solution to the (backward) Cauchy problem
	\begin{equation*}
		\begin{cases}
			{\mathfrak h}' = \mathfrak A\mathfrak h  \quad \text{on }(0,s_*),\\
			\mathfrak h(s_*) = 0.
		\end{cases}
	\end{equation*}
	This implies that $\mathfrak h\equiv 0$ on $(0,s_*)$ and thus $h\equiv 0$ on  $(t_*,t_1)$, which contradicts the definition of the interval $(t_0,t_1)$.
	
	To complete the proof of the statement, observe that $t\mapsto h(t)$ is bounded on $[t_*,t_1]$ and thus belongs to $L^\infty((t_*, t_1),\R^2)$.
	Then, for every  $p\ge 1$,
	\begin{align*}
		\int_0^{+\infty}|\mathfrak h(s)|^p\,ds &= \int_{t_*}^{t_1}|P^{-1}h|^p|h|^{-1}\,dt \leq \|P^{-1}\|^{p}  \int_{t_*}^{t_1}|h|^{p-1}\,dt \\
		& \leq \|P^{-1}\|^{p} \|h\|_{L^{\infty}}^{p-1} (t_{1}-t_{*})<+\infty.
		\qedhere
	\end{align*}
\end{proof}

\subsection{Proof of Proposition~\ref{prop:h/h} in the case $\det A(t_1)<0$}

	Since $\tr(A)=0$ and $\det A(t_1)<0$, there exists $P\in {\rm GL}(2,\R)$	such that
	\begin{equation*}
		P A(t_1) P^{-1} =
		\left(
		\begin{array}{cc}
			-a & 0 \\
			0 & a
		\end{array}
		\right), \qquad a>0.
	\end{equation*}
	Up to applying the change of coordinates associated with $P$ and defining the  time-rescaled curves $\mathfrak h$ and $\mathfrak  A$ as in Lemma~\ref{lem:s}, we have
	\begin{equation}
		\mathfrak A(s) = \left(
		\begin{array}{cc}
			-\alpha(s) & \beta(s) \\
			\zeta(s) & \alpha(s)
		\end{array}
		\right),
	\end{equation}
	where
	\begin{equation}\label{eq:limst}
	\lim_{s\to\infty}\alpha(s)=a, \quad \lim_{s\to\infty}\zeta(s)=\lim_{s\to\infty}\beta(s)=0.
	\end{equation}
	Let $\mathfrak  h=\rho e^{i\theta}$ for $\rho>0$ and $\theta\in[0,2\pi)$.
We will prove that $\theta (s) \rightarrow 0 \mod \pi$ as $s\to\infty$.

	Observe that, letting $\mathfrak h=(x_1,x_2)$ with $x_1,x_2\in\R$, we have
	\begin{equation}\label{eq:tan}
		\frac12\tan 2\theta = \frac{\sin\theta\cos\theta}{\cos^2\theta-\sin^2\theta} = \frac{x_1x_2}{x_1^2-x_2^2}.
	\end{equation}
	By \eqref{eq:A'h}
	and simple computations we obtain
	\begin{gather}
		(x_1x_2)' = \zeta x_1^2 + \beta x_2^2, \qquad
		\frac{(x_1^2-x_2^2)' }{2} = -\alpha (x_1^2+x_2^2) + (\beta-\zeta)x_1x_2.
	\end{gather}
	Upon integration and exploiting \eqref{eq:limst},
	we get
	\begin{equation} \label{eq:R}
		x_1x_2 = o(R), \quad x_1^2-x_2^2 = 2a R(1+o(1)), \quad \text{where}\quad R(s) := \int_s^{+\infty} |\mathfrak h(\sigma)|^2\,d\sigma.
	\end{equation}
	Observe that, by Lemma~\ref{lem:s}, $\mathfrak h\in L^2((0,+\infty),\R^2)$ and, in particular, $R\rightarrow0$ as $s\rightarrow+\infty$.
	Finally, substituting the above in \eqref{eq:tan} shows that $\tan2\theta\rightarrow 0$.
 From the second equation in \eqref{eq:R}, the sign of $x_1^2-x_2^2$ is positive as $t \uparrow t_1$, which implies that $\theta \to 0 \mod \pi$. This completes the proof
of Proposition~\ref{prop:h/h} in the case $\det A(t_1)<0$.

\begin{rmk} \label{r:eigenvectors}
Recall that in the analysis above we  suppose that $u=\frac{h}{|h|}$, and we actually prove that in this case $u(t)$ converges to a unit eigenvector of $A(t_1)$ associated with the negative eigenvalue $-a$.
In the case where $u=-\frac{h}{|h|}$, 
an analogous argument yields
that  $u(t)$ converges to a unit eigenvector of $A(t_1)$ associated with the positive eigenvalue $a$.
\end{rmk}


\subsection{Proof of Proposition~\ref{prop:h/h} in the case $\det A(t_1)=0$}\label{s:det0}

Assume that  $\det A(t_1)=0$ and recall that $\tr A(t_1)=0$. Since, moreover, $A(t_1)\ne0$, there exists $P\in {\rm GL}(2,\R)$ such that
\begin{equation}\label{eq:Pnilp}
	P A(t_1) P^{-1} =
	\left(
	\begin{array}{cc}
		0 & 1 \\
		0 & 0
	\end{array}
	\right).
\end{equation}
As before, using the change of variables of Lemma~\ref{lem:s},
we let
	\begin{equation*}
		\mathfrak  A(s) = \left(
		\begin{array}{cc}
			-\alpha(s) & \beta(s) \\
			\zeta(s) & \alpha(s)
		\end{array}
		\right),
	\end{equation*}
	where $\alpha,\beta,\zeta$ are linear combinations of $h_{2112}\circ \varphi^{-1}, h_{2212}\circ \varphi^{-1}$, and $h_{1112}\circ \varphi^{-1}$, and hence absolutely continuous with bounded derivatives on $(0,+\infty)$, according to \eqref{eq:poisson}. Equality~\eqref{eq:Pnilp} implies that $\alpha\rightarrow 0$, $\beta\rightarrow 1$, and $\zeta\rightarrow 0$ as $s\to+\infty$. 	
	
	We also introduce $\mu := \zeta+\beta$ and we notice that $\mu \rightarrow 1$ as $s\to+\infty$.
	(Beware that the same letters are used for different parameters in the proof of Proposition~\ref{prop:nonpos}.)
	Then, \eqref{eq:A'h} reads
	\begin{equation*}
		\frac{\rho'}{\rho}=\mu\sin\theta\cos\theta-\alpha\cos2\theta, \qquad \theta'=-\mu\sin^2\theta+\alpha\sin2\theta+\zeta,
	\end{equation*}
and can be written as
	\begin{equation}\label{eq:polar2}
		\frac{\rho'}{\rho}=\sin\theta\cos\theta+f, \qquad \theta'=-\sin^2\theta+g,	
	\end{equation}
where the functions
	\begin{equation*}
		f = -\alpha\cos2\theta + (\mu-1)\sin\theta\cos\theta, \qquad g=\alpha\sin2\theta+\zeta +(1-\mu)\sin^2\theta,
	\end{equation*}
 tend to zero as $s\to+\infty$.
	
Establishing Proposition~\ref{prop:h/h}, finally amounts to proving that $\theta\rightarrow 0 \mod \pi$, as $s\rightarrow+\infty$.
\begin{lemma}\label{lem:turning}
  We have the following dichotomy:
  \begin{enumerate}[label=\rm(\roman*)]
    \item \label{p-i}
    $\theta\rightarrow 0 \mod \pi$, as $s\rightarrow+\infty$;
        \item\label{p-ii}
    $\theta\rightarrow-\infty$ as $s\rightarrow+\infty$. Moreover, in this case,
    for any $0<\varepsilon<\pi/2$
    there exists an increasing sequence of positive real numbers $(s_n)_{n\in\mathbb N}$ 
    tending to infinity  such that
    \begin{gather*}
      \theta(s_{2n}) = \pi-\varepsilon  \mod 2\pi, \qquad   \theta(s_{2n+1}) = \varepsilon \mod 2\pi,\\
      \theta'(s)<0 \quad \forall s\in [s_{2n},s_{2n+1}]. 
    \end{gather*}
  \end{enumerate}
\end{lemma}

\begin{figure}[h]
\includegraphics[width=10cm]{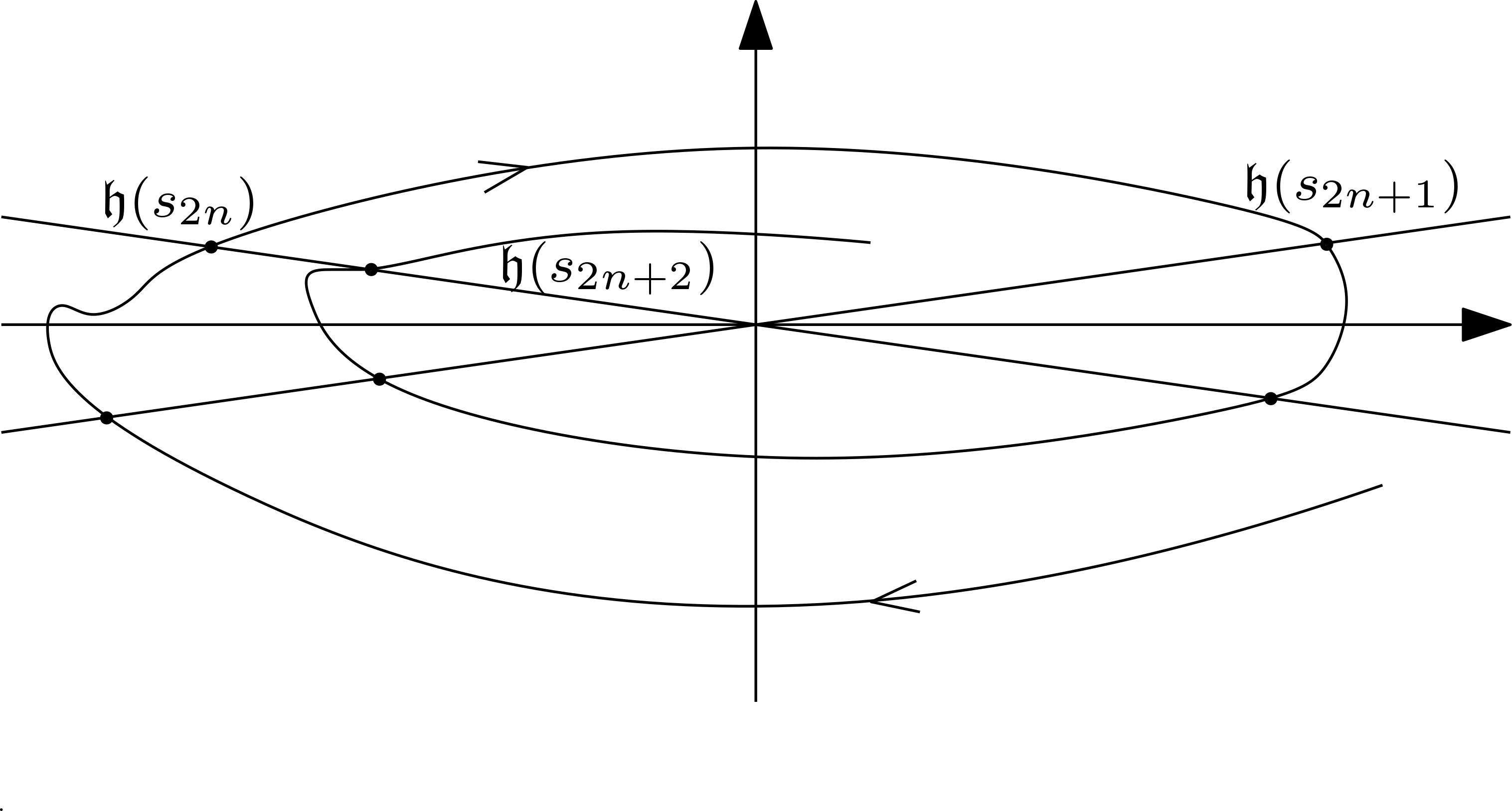}
\caption{The sequence $(s_n)_{n\in\mathbb N}$  in Lemma~\ref{lem:turning}}
\label{f:turn}
\end{figure}

\begin{proof}
   Notice that the dynamics of $\theta$ is a perturbation via $g$ of
  \begin{equation*}
    \theta' = -\sin^2\theta.
  \end{equation*}
The phase portrait of the latter on $\mathbb S^1$ is made of two equilibria in $0$ and $\pi$ joined by two clock-wise oriented heteroclinic trajectories.

  Assume that \ref{p-i} does not hold. Therefore, there exists  $c>0$ such that
  \begin{equation}\label{eq:c}
  \limsup_{s\to+\infty} |\sin\theta(s)|>c.
  \end{equation} Let $\eps>0$ be such that $\sin\eps\in (0,c)$ and  $s^*>0$ be such that, for $s>s^*$, $\vartheta'(s)<-\eps^2/2$ as soon as $|\sin\theta(s)|> \eps$.

 Pick $q_1>s^*$  such that $|\sin\theta(q_1)|>c> \sin\eps$.
 Since $\vartheta'$ is bounded from zero as long as  $|\sin\theta|$ stays larger than $\sin\eps$, there exists $r_1>q_1$ such that $|\sin\theta(r_1)|= \sin\eps$. By definition of $c$, there exists $q_2>r_1$ such that $|\sin\theta(q_2)|> c$. Moreover, $q_1$ and $q_2$ can be taken so that $\theta(q_2)=\theta(q_1)-\pi$  and \eqref{eq:c} holds with $c$ arbitrarily close to $1$. By iterating the procedure leading from $q_1$ to $q_2$, we prove that $\theta\rightarrow-\infty$. The construction also shows how to define  the sequence $(s_n)_{n\in\mathbb{N}}$ as in
 \ref{p-ii} (cf.\ Figure~\ref{f:turn}).
\end{proof}

The rest of the argument consists in showing that case \ref{p-ii} in Lemma~\ref{lem:turning} cannot  hold true. For that purpose, we argue by contradiction.

\begin{lemma}
\label{lem:estimate0}
Assume that property \ref{p-ii} in Lemma~\ref{lem:turning} holds true.
Then there exists $0<\varepsilon_0<\pi/2$ such that for any $0<\varepsilon<\varepsilon_0$ there exists $N_\varepsilon$ for which, given any  $n\geq N_\varepsilon$,
 \begin{equation}\label{eq:basic1}
\frac2{\varepsilon}\left(1-{\varepsilon^2} \right)
\le s_{2n+1}-s_{2n}\le \frac2\varepsilon\left(1+{\varepsilon^2} \right),
\end{equation}
and
\begin{equation}\label{eq:basic2}
(1-\varepsilon)\varepsilon\rho(s_{2n})\leq \rho(s)\sin\theta(s)\leq (1+\varepsilon)\varepsilon\rho(s_{2n}),\hbox{ for }s\in[s_{2n},s_{2n+1}].
\end{equation}
As a consequence, for every $n\ge N_\eps$, one has the following estimates
\begin{align}
 2(1-2\varepsilon)\rho(s_{2n})\leq
\int_{s_{2n}}^{s_{2n+1}} \sin\theta(s)\rho(s)\,ds&\leq 2(1+2\varepsilon)\rho(s_{2n}),\label{esti1}\\
 (1-2\varepsilon)\frac{\rho(s_{2n})}{\varepsilon}\leq \int_{s_{2n}}^{s_{2n+1}}\rho(s)\,ds&\leq(1+2\varepsilon)\frac{\rho(s_{2n})}{\varepsilon},\label{esti2}\\
 \Big\vert\int_{s_{2n}}^{s_{2n+1}} \cos\theta(s)\rho(s)\,ds\Big\vert&\leq
 \rho(s_{2n}).\label{esti3}
 \end{align}
\end{lemma}

\begin{proof}
Set
$M_f(s)= \sup_{\tau\ge s}|f(\tau)|$ and $M_g(s)= \sup_{\tau\ge s}|g(\tau)|$. Observe that these two functions tend to zero as $s$ tends to infinity.

By Lemma~\ref{lem:turning}, for $n$ large enough and $s\in[s_{2n},s_{2n+1}]$, equation \eqref{eq:polar2} becomes
  \begin{equation}\label{eq:cotan}
    (\cot\theta)' = 1-\frac{g}{\sin^2\theta}.
  \end{equation}
For $n$ large enough, for every $s\in[s_{2n},s_{2n+1}]$ we have
  \begin{equation*}
    \left|\frac{g(s)}{\sin^2\theta(s)}\right|\le \frac{M_g(s_{2n})}{\sin^2\eps} \le \frac{\varepsilon^2}{2}.
  \end{equation*}
Equation  \eqref{eq:basic1} follows by integrating \eqref{eq:cotan} on the interval $[s_{2n},s_{2n+1}]$.

On the interval $[s_{2n},s_{2n+1}]$, one has
$$
\frac{\rho'(s)}{\rho(s)}+\frac{\theta'(s)\cos\theta(s)}{\sin\theta(s)}=f(s)+\frac{\cos\theta(s) g(s)}{\sin\theta(s)}.
$$
For $n$ large enough, for every $s\in[s_{2n},s_{2n+1}]$ we have
\begin{equation}\label{eq:int00}
\left\vert f(s)\right\vert+ \left\vert\frac{\cos\theta(s) g(s)}{\sin\theta(s)}\right\vert
\leq M_f(s_{2n})+
\frac{M_g(s_{2n})}{\sin\varepsilon}\leq \frac{\varepsilon^2}4.
\end{equation}
By integrating between $s_{2n}$ and any $s\in [s_{2n},s_{2n+1}]$, one gets
\begin{align*}
\left\vert \ln\left(\frac{\rho(s)\sin\theta(s)}{ \rho(s_{2n})\sin\varepsilon}\right)\right\vert
&=\left\vert\int_{s_{2n}}^s
\left(f(s)+\frac{\cos\theta(s) g(s)}{\sin\theta(s)}\right) ds\right\vert\leq \frac{(s_{2n+1}-s_{2n})\varepsilon^2}4\leq \frac{\varepsilon}{2}(1+\eps^{2}),
\end{align*}
yielding \eqref{eq:basic2} for $\varepsilon$ small enough.

We now turn to the proof of the three estimates \eqref{esti1}--\eqref{esti3}.
The first one simply follows by integrating  \eqref{eq:basic2} on $[s_{2n},s_{2n+1}]$ and using \eqref{eq:basic1}. Estimate~\eqref{esti2} is obtained by first dividing
 \eqref{eq:basic2} by $\sin\theta(s)$ and then integrating the resulting inequalities on $[s_{2n},s_{2n+1}]$. One gets that
  $$
  (1-\varepsilon)\varepsilon\rho(s_{2n})\int_{s_{2n}}^{s_{2n+1}} \frac{ds}{\sin\theta(s)}\leq
\int_{s_{2n}}^{s_{2n+1}} \rho(s)\,ds
\leq (1+\varepsilon)\varepsilon\rho(s_{2n})\int_{s_{2n}}^{s_{2n+1}} \frac{ds}{\sin\theta(s)}.
$$
 On the other hand, the following holds true,
 $$
 \int_{s_{2n}}^{s_{2n+1}} \frac{ds}{\sin\theta(s)}=\int_{s_{2n}}^{s_{2n+1}} \frac{\theta'(s)}{-\sin^3\theta(s)\big(1-\frac{g(s)}{\sin^2\theta(s)}\big)}\,ds,
 $$
 which implies that
 $$
 (1-\varepsilon^2)\int_\varepsilon^{\pi-\varepsilon}\frac{d\theta}{\sin^3\theta}\leq \int_{s_{2n}}^{s_{2n+1}} \frac{ds}{\sin\theta(s)}\leq (1+\varepsilon^2)\int_\varepsilon^{\pi-\varepsilon}\frac{d\theta}{\sin^3\theta}.
 $$
 A direct computation shows that $\int_\varepsilon^{\pi-\varepsilon}\frac{d\theta}{\sin^3\theta}=\frac1{\varepsilon^2(1+o(\varepsilon))}$ as $\varepsilon$ tends to zero. One finally deduces estimate~\eqref{esti2}.

 To derive estimate~\eqref{esti3},
 one notices that
 \begin{align*}
 \int_{s_{2n}}^{s_{2n+1}} \cos\theta(s)\rho(s)\,ds&=\int_{s_{2n}}^{s_{2n+1}} \frac{\sin\theta(s)\cos\theta(s)\rho(s)}{\sin\theta(s)}\,ds\\
 &=
 \int_{s_{2n}}^{s_{2n+1}} \frac{\rho'(s)
 -f(s)\rho(s)}{\sin\theta(s)}\,ds\\
 &= -\int_{s_{2n}}^{s_{2n+1}} \frac{f(s)\rho(s)}{\sin\theta(s)}\,ds+
 \frac{\rho(s_{2n+1})-\rho(s_{2n})}{\sin\varepsilon}\\&\qquad +\int_{s_{2n}}^{s_{2n+1}} \frac{\rho(s)\cos\theta(s)\theta'(s)}{\sin^2\theta(s)}\,ds.
 \end{align*}
 By using the expression of $\theta'$ in the last integral, one deduces that
 $$
 2\int_{s_{2n}}^{s_{2n+1}} \cos\theta(s)\rho(s)\,ds= \frac{\rho(s_{2n+1})-\rho(s_{2n})}{\sin\varepsilon}-\int_{s_{2n}}^{s_{2n+1}}\rho(s)\frac{f(s)-\frac{\cos\theta(s) g(s)}{\sin\theta(s)}}{\sin\theta(s)}\,ds.
 $$
 By using \eqref{eq:basic2} for $s=s_{2n}$ and $s=s_{2n+1}$ and then \eqref{eq:int00}, one deduces \eqref{esti3}.
 \end{proof}

Fix a sequence $(\varepsilon_k)_{k\in\mathbb N}$,
strictly decreasing to $0$.
For each $k\in\mathbb N$, we use $(s_{k,n})_{n\in\mathbb N}$
to denote the sequence $(s_n)_{n\in\mathbb N}$
given by Lemma~\ref{lem:turning} and corresponding to $\varepsilon=\varepsilon_k$. For all $k\in\mathbb N$ let $n_k\ge N_{\varepsilon_k}$ be an integer to be fixed later, where $N_{\varepsilon_k}$ is as in Lemma~\ref{lem:estimate0}. We use $(\xi_\ell)_{\ell\in\mathbb N}$
to denote the sequence defined by
\begin{equation*}
  \xi_{2k} = s_{k,2n_k},\quad \xi_{2k+1} = s_{k,2n_k+1},
  \qquad \forall k\in\mathbb N.
\end{equation*}
We choose $k\mapsto n_k$ so that the sequence $(\xi_\ell)_{\ell\in\mathbb N}$ is strictly increasing and tends to  infinity
as $\ell\to+\infty$.

Let $t_\ell= \varphi^{-1}(\xi_\ell)$, where $\varphi$ is the change of variables introduced in Lemma~\ref{lem:s}. For every $\ell\ge 0$ consider the function $u_\ell\in L^\infty([0,1],\mathbb S^1)$ defined by $u_\ell(\tau)=u(t_{2\ell}+\tau(t_{2\ell+1}-t_{2\ell}))$.
By the weak-$\star$ compactness of all bounded subsets of $L^\infty([0,1],\mathbb R^2)$,
we can assume without loss of generality that $u_\ell\rightharpoonup u_\star$ in the weak-$\star$ topology.
Applying Lemma~\ref{lem:nilpot}
with $a_\ell = t_{2\ell}$ and $b_\ell=t_{2\ell+1}$, we deduce that $u_\star$ is minimizing and $|u_\star|\equiv 1$ almost everywhere in $[0,1]$.

For every subinterval $[a,b]$ of $[0,1]$,
by the properties of weak-$\star$ convergence, we have that
\begin{equation*}
  \int_a^b v^{T}
  u_\ell(\tau)\,d\tau \rightarrow   \int_a^b  v^{T} u_\star(\tau) \,d\tau ,\qquad \forall\, v\in \R^{2}.
\end{equation*}
Moreover, one has
\begin{align}\nonumber
  \int_a^b v^{T}
  u_\ell(\tau)\,d\tau
    &=\frac{1}{t_{2\ell+1}-t_{2\ell}}\int_{(1-a)t_{2\ell}+at_{2\ell+1}}^{(1-b)t_{2\ell}+bt_{2\ell+1}} \frac{v^{T} h}{|h|}\,dt \\
    &= \frac{1}{t_{2\ell+1}-t_{2\ell}}\int_{\varphi((1-a)t_{2\ell}+at_{2\ell+1})}^{\varphi((1-b)t_{2\ell}+bt_{2\ell+1})} v^{T} P\mathfrak h(s)\,ds,\label{eq:fon0}
\end{align}
where $P$ has been introduced in \eqref{eq:Pnilp}.

In addition
\begin{equation} \label{eq:tp}
  t_{2\ell+1}-t_{2\ell} = \int_{\xi_{2\ell}}^{\xi_{2\ell+1}} |P\mathfrak  h(s)|\,ds.
\end{equation}

\begin{lemma}\label{lem8}
  Under the above assumptions, there exists a  unit vector $v_{\star}\in \R^{2}$ such that $u_\star (t)= v_{\star}$ for a.e.\  $t\in [0,1]$.  Moreover, $v_\star$ is parallel to $P(1,0)$.
\end{lemma}

\begin{proof}
 Let  $v_\star,w_\star
\in\R^2$ be two
orthogonal unit vectors such that $v_\star$ is
parallel to $P(1,0)$. Notice that  $P^{T}w_\star$ is orthogonal to $(1,0)$, that is, it is parallel to $(0,1)$.
  We start by showing that $ w_\star^{T}u_\star(t)=0$ for a.e.\  $t\in [0,1]$.
  This amounts to showing that for all $0\le a<b\le 1$ it holds
  \begin{equation*}
   \frac{1}{t_{2\ell+1}-t_{2\ell}}\int_{\varphi((1-a)t_{2\ell}+at_{2\ell+1})}^{\varphi((1-b)t_{2\ell}+bt_{2\ell+1})} \mathfrak h_2(s)\,ds  \rightarrow 0 \qquad\text{as }\ell\to+\infty.
  \end{equation*}
  Since $\mathfrak h_2=\rho\sin\theta$ is positive on $[\xi_{2\ell},\xi_{2\ell+1}]$ by construction, using \eqref{eq:tp} it is enough to show that
\begin{equation*}
    \frac{\int_{\xi_{2\ell}}^{\xi_{2\ell+1}} \rho(s)\sin\theta(s)\,ds}{\int_{\xi_{2\ell}}^{\xi_{2\ell+1}} |P\mathfrak  h(s)|\,ds}\longrightarrow 0 \qquad\text{as }\ell\to+\infty.
  \end{equation*}
Since $|P\mathfrak  h(s)|\geq \|P^{-1}\|^{-1}\rho(s)$ for all $s$, the latter limit holds true according to \eqref{esti1} and \eqref{esti2} in Lemma~\ref{lem:estimate0}, applied to $\varepsilon=\varepsilon_\ell$  for $\ell\geq 0$.

Recall that the control $u_\star$  is minimizing  and $|u_\star(t)|= 1$ for a.e.\ $t\in [0,1]$. From what precedes, one deduces that $u_\star$ is almost everywhere perpendicular to  $w_\star$, hence equal to $v_\star$ or $-v_\star$. It then follows from  Lemma~\ref{lem:constantsign} in the appendix that, up to replacing $v_\star$ by $-v_\star$, the equality $u_\star(t)=v_{\star}$ holds for a.e.\ $t\in [0,1]$.
\end{proof}

Let $ \bar v\in \R^{2}$  be such that $P^{T} \bar v=(1,0)$. We have, according to Lemma~\ref{lem8},
\begin{equation*}
 \lim_{\ell\to\infty} \int_0^1   \bar v^{T} u_\ell(\tau)\,d\tau  = \int_0^1   \bar v^{T} u_\star(\tau) \,d\tau=  \bar v^{T}v_{\star} \neq 0.
\end{equation*}
We conclude the proof by contradiction by showing that the limit in the left-hand side is zero. Indeed, according to  \eqref{eq:fon0}, we have
\begin{equation*}
\left|\int_0^1   \bar v^{T} u_\ell(\tau)\,d\tau\right| =  \frac{\left|\int_{\xi_{2\ell}}^{\xi_{2\ell+1}} \rho(s)\cos\theta(s)\,ds \right|}{\int_{\xi_{2\ell}}^{\xi_{2\ell+1}} |P\mathfrak  h(s)|\,ds} \leq \|P^{-1}\|\frac{\left|  \int_{\xi_{2\ell}}^{\xi_{2\ell+1}} \rho(s)\cos\theta(s)\,ds \right|
 }{\int_{\xi_{2\ell}}^{\xi_{2\ell+1}} \rho(s)\,ds}.
 \end{equation*}
The right-hand side of the above equation tends to zero thanks to  \eqref{esti2} and \eqref{esti3} in Lemma~\ref{lem:estimate0} applied to $\varepsilon=\varepsilon_\ell$  for $\ell\geq 0$.

We have therefore proved that \ref{p-ii} in Lemma~\ref{lem:turning} cannot hold true, which
 completes the proof of Proposition~\ref{prop:h/h}.

\section{Proof of Theorem~\ref{t:main}}\label{proof-ex-conj}

Let $M$ be as in the statement of Theorem~\ref{t:main}.
Denote, as in the previous sections, by $\gamma:[0,T]\to M$ a length-minimizing trajectory parametrized by arclength and
by $\lambda:[0,T]\to T^*M$ an abnormal extremal lift  of $\gamma$.

 Proposition~\ref{prop:h/h}, together with Theorem~\ref{t:HLK},
  proves the $C^1$ regularity of $\gamma$ provided that $h$ vanishes only at isolated points.

We consider in this section the case where $t_0\in (0,T)$
is a density point of
$\{ t\in [0,T]\mid h(t)=0\}$.
We want to prove that
$u(t)$ (up to modification on a set of measure zero) has a limit as $t\uparrow t_0$ and as $t\downarrow t_0$.
By symmetry, we restrict our attention to the existence of
 the limit of $u(t)$ as    $t\uparrow t_0$.

We are going to consider separately the situations where $h\equiv 0$ on a left neighborhood of $t_0$
and where
there exists a sequence of maximal open intervals $(t_0^n,t_1^n)$ with
$h|_{(t_0^n,t_1^n)}\ne 0$
and  such that $t_1^n\rightarrow t_0$.

Assume for now on that $h\equiv 0$ on a left neighborhood $(t_0-\eta,t_0]$ of $t_0$. Then, since $\dot h =Au$ almost everywhere on $(t_0-\eta,t_0]$, we have that $u(t)$ belongs to $\ker A(t)$ for almost every $t$ in $(t_0-\eta,t_0]$.
By Lemma~\ref{l:dis}, moreover,
 $\ker A(t)$ is one-dimensional for every $t\in (t_0-\eta,t_0]$.

Fix an open neighborhood $V_{0}$ of $\lambda(t_0)$ in $T^{*}M$ such that there exists a smooth map $V_{0}\ni \lambda\mapsto v(\lambda) \in \mathbb S^{1}$ such that $v(\lambda(t))\in \ker A(t)$ if $\lambda(t)\in V_{0}$
and $t\in (t_0-\eta,t_0]$. Up to reducing $\eta$, we assume that $\lambda(t)\in V_{0}$ for every $t\in (t_0-\eta,t_0]$.
Notice that $\lambda|_{(t_0-\eta,t_0]}$ is a solution of the time-varying system
$$\dot \lambda = \sigma(t) \vec{X}_{v(\lambda)}(\lambda),$$
where $\sigma:(t_0-\eta,t_0]\to  \{-1,1\}$ is measurable. Hence, by length-minimality of $\gamma$ and by Lemma~\ref{lem:constantsign} in the appendix, either $u=v$ almost everywhere on $(t_0-\eta,t_0]$ or $u=-v$ almost everywhere on $(t_0-\eta,t_0]$. We conclude that $u$ is continuous on $(t_0-\eta,t_0]$ and the
proof in this case in concluded.

We are left to consider the case where
every left neighborhood of $t_0$ contains a maximal interval $(\tau_0,\tau_1)$ such that
$h\ne 0$ on $(\tau_0,\tau_1)$.

Notice that, by Proposition~\ref{prop:nonpos} and by continuity of $t\mapsto A(t)$, we have  that $\det A(t_0)\le 0$.

The case
$\det A(t_0)<0$ can be ruled out
thanks to the following lemma.

\begin{lemma}
Let $\det A(t_0)<0$.
There exists $\eta\in (0,t_0)$ such that, for any maximal interval $(\tau_0,\tau_1) \subset (0,t_0)$ on which
$h(t)\ne 0$,
then $\tau_0<t_0-\eta$.
\end{lemma}
\begin{proof}
As we have already seen in Section~\ref{s:intro_dyn}, on every interval
where $h(t)\neq 0$ we have that $u$ is smooth and equal to  either
$\frac{h(t)}{|h(t)|}$ or $-\frac{h(t)}{|h(t)|}$. Thus the function $h$ on $(\tau_0,\tau_1)$ is either a maximal solution to $\dot x=A(t)\frac{x}{|x|}$ or a maximal solution to $\dot x=-A(t)\frac{x}{|x|}$. Let us assume that it is a maximal solution of $\dot x=A(t)\frac{x}{|x|}$, the proof being identical in the second case.

For every $v\in \mathbb{R}^2\setminus\{0\}$ and every $\theta>0$ denote by $C_{\theta}(v)$ the cone of all vectors in
$\mathbb{R}^2\setminus\{0\}$ making an (unoriented) angle smaller than $\theta$ with $v$ or $-v$.

Let $\eta_0\in (0,t_0)$ be such that $\det(A(t))<0$  for every $t\in[t_0-\eta_0,t_0]$.
For $t\in[t_0-\eta_0,t_0]$, denote by $v_{-}(t)$ and
$v_{+}(t)$ two unit eigenvectors of $A(t)$, the first corresponding to a negative and the second to  a positive eigenvalue.

Let $\eta\in (0,\eta_0)$ and $\theta_0>0$ be such that $C_{\theta_0}(v_{+}(t_0))\cap C_{\theta_0}(v_{-}(t_0))= \emptyset$ and
$v_\pm(t)\in C_{\theta_0}(v_{\pm}(t_0))$
for every $t\in [t_0-\eta,t_0]$.
Notice that, for every fixed $\bar t\in [t_0-\eta,t_0]$, the vector field $x\mapsto  A(\bar t)x$  points inward $C_{\theta_0}(v_{+}(t_0))$ at every nonzero point of its boundary (see Figure~\ref{fig:hyperbolics}). Hence $C_{\theta_0}(v_{+}(t_0))$
is positively invariant for the dynamics of $\dot x=A( t)\frac{x}{|x|}$ on $[t_0-\eta,t_0]$.

\begin{figure}
\includegraphics[scale=.07]{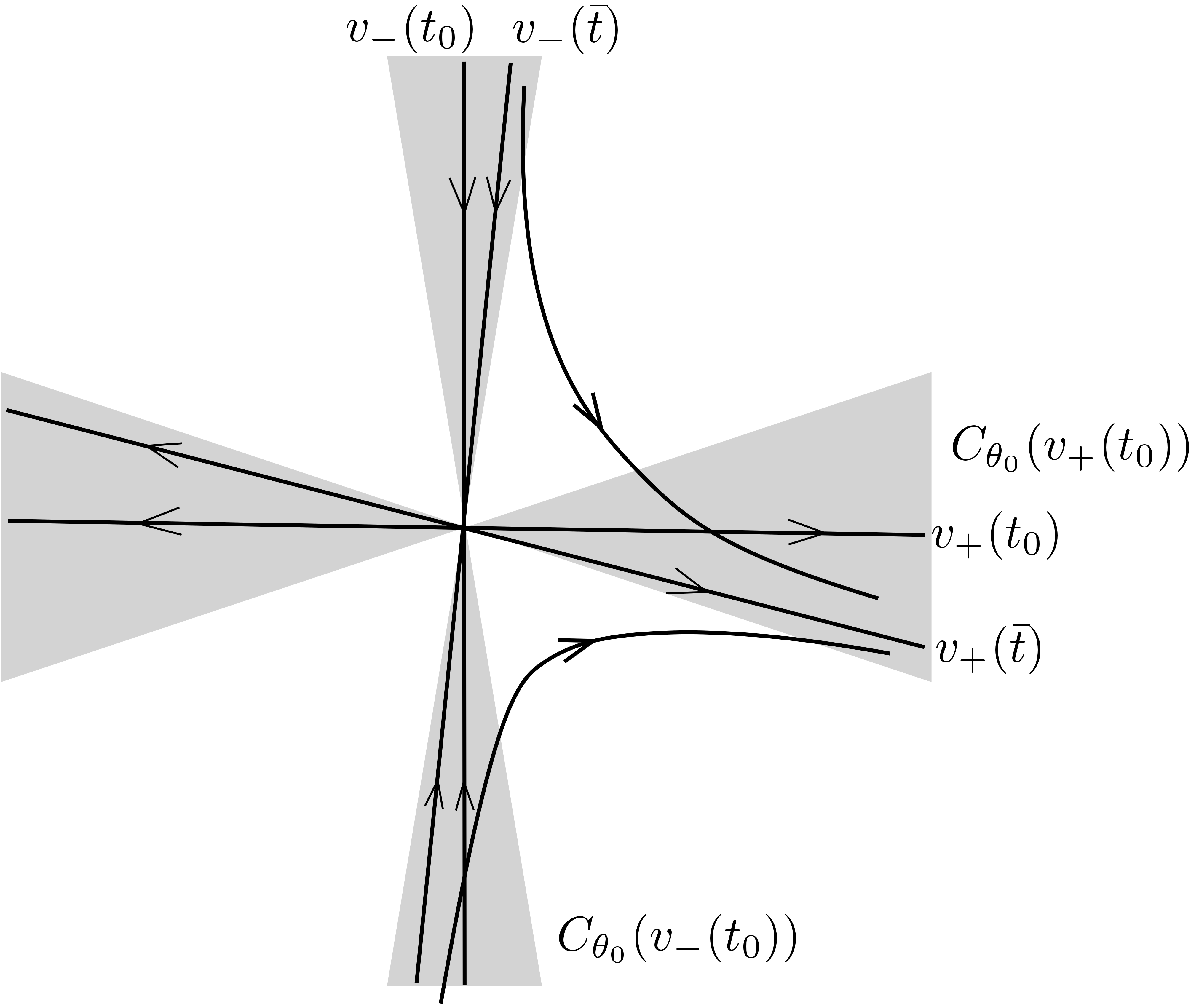}
\caption{Phase portrait of $\dot x=A(\bar t)x$ for $\bar{t}\in [t_0-\eta,t_0]$}
\label{fig:hyperbolics}
\end{figure}

In order to prove the statement, we argue by contradiction. Assume that
$h:(\tau_{0},\tau_{1})\to \R^{2}\setminus \{0\}$ is a maximal solution
 of $\dot x=A(t)\frac{x}{|x|}$ with $(\tau_{0},\tau_{1})\subset (t_0-\eta,t_0)$. Then $h(\tau)$ tends to $0$ as $\tau$ tends to $\tau_0$ or $\tau_1$ and it follows from Proposition~\ref{prop:h/h} that $\frac{h(\tau)}{|h(\tau)|}$ converges to an eigenvector of $A(\tau_0)$ as $\tau\downarrow\tau_0$ and to an eigenvector of $A(\tau_1)$ as $\tau\uparrow\tau_1$. More precisely, from Remark~\ref{r:eigenvectors} there holds
$$\lim_{\tau\downarrow\tau_0}\frac{h(\tau)}{|h(\tau)|}\to \pm v_{+}(\tau_0),\quad \lim_{\tau\uparrow\tau_1}\frac{h(\tau)}{|h(\tau)|}\to \pm v_{-}(\tau_1).$$
This contradicts the positive invariance of
$C_{\theta_0}(v_{+}(t_0))$ for the equation $\dot x=A(t)\frac{x}{|x|}$ on $(\tau_{0},\tau_{1})$.
\end{proof}

In the case
$\det A(t_0)=0$ the proof follows the steps of the construction of Section~\ref{s:det0}. In particular, let $P\in {\rm GL}(2,\R)$ be such that
\begin{equation*}
		  P^{-1}A(t)P = \left(
		\begin{array}{cc}
			-a(t) & b(t) \\
			c(t) & a(t)
		\end{array}
		\right),
	\end{equation*}
	where $a,b,c$ are affine combinations of $h_{2112}, h_{2212}$, and $h_{1112}$
	with $a\rightarrow 0$, $b\rightarrow 1$, and $c\rightarrow 0$ as $t\to t_0$. 	
Let
\begin{equation}\label{eq:hrw}
P^{-1}h(t)=r(t)e^{i\omega(t)},
\end{equation}
 with $\omega(t)$ uniquely defined modulus $2\pi$ only when $h(t)\ne 0$.
	
The crucial point is the following counterpart to
Lemma~\ref{lem:turning}, whose proof can be obtained using exactly the same arguments.

\begin{lemma}\label{lem:turning-second}
  We have the following dichotomy:
  \begin{enumerate}[label=\rm(\roman*)]
   \item\label{p-iii2}     for any $0<\varepsilon<\pi/2$, for $\eta$ small enough,
 $|\sin(\omega(t))|<\varepsilon$ for all $t\in (t_0-\eta,t_0)$ such that $h(t)\ne0$;
    \item\label{p-ii2}     for any $0<\varepsilon<\pi/2$
    there exists an increasing sequence $(t_n)_{n\in\mathbb N}$ in $(0,t_0)$
    tending to $t_0$ and such that
    \begin{gather*}
      \omega(t_{2n}) = \pi-\varepsilon \mod 2\pi, \qquad \omega(t_{2n+1}) = \varepsilon \mod 2\pi,\\
     h(t)\ne 0,\ \sin(\omega(t))>0,\ \dot\omega(t)<0 \quad \forall t\in [t_{2n},t_{2n+1}]
    \end{gather*}
 or
   \begin{gather*}
      \omega(t_{2n}) = -\varepsilon \mod 2\pi, \qquad \omega(t_{2n+1}) = \varepsilon-\pi  \mod 2\pi,\\
      h(t)\ne 0,\ \sin(\omega(t))<0,\  \dot\omega(t)<0 \quad \forall t\in [t_{2n},t_{2n+1}]
    \end{gather*}
    holds true.
 \end{enumerate}
\end{lemma}

Case \ref{p-ii2} can be excluded by similar computations as in Section~\ref{s:det0}, since it contradicts the optimality of $\gamma$.

Consider now case \ref{p-iii2}.
Let  $v_\star,w_\star
\in\R^2$ be two
orthogonal unit vectors such that $v_\star$ is
parallel to $P(1,0)$.
 According to \eqref{eq:hrw},  if $\sin(\omega(t))=0$ and $r(t)\ne 0$, then $u(t)=h(t)/|h(t)|$ is equal to $v_\star$ or $-v_\star$.
For every $\eta\in (0,t_0)$ we set
   \[ I^+_\eta=\{t\in (t_0-\eta,t_0) \mid v_\star^Tu(t)>0\}, \quad I^-_\eta=\{t\in(t_0-\eta,t_0)  \mid v_\star^Tu(t)<0\}.\]
Property \ref{p-iii2} implies that, for $\eta$ small, $I^+_\eta\cup I^-_\eta$ contains
$\{t\in (t_0-\eta,t_0)\mid h(t)\ne 0\}$.
Moreover, if $t_0$ is a density point for $I=I^+_\eta\cap \{t\in (t_0-\eta,t_0)\mid h(t)\ne 0\}$ (respectively, $I=I^-_\eta\cap \{t\in (t_0-\eta,t_0)\mid h(t)\ne 0\}$), then,
\begin{equation*}
    \lim_{t\in I,t\rightarrow t_0} u(t) = v_\star
    \qquad
    (\text{respectively, }   \lim_{t\in I,t\rightarrow t_0} u(t) = -v_\star).
\end{equation*}

Let $\Phi_\eta=\{t \in (t_0-\eta,t_0)\mid h(t)=0\}$. For almost every $t\in \Phi_\eta$,  $u(t)$ is in the kernel of $A(t)$ and $|u(t)|=1$.
Notice that, if $t_0$ is a density point for $J=\{t\in (0,t_0)\mid \ker A(t)\ne (0)\}$, then the kernel of $A(t)$ converges to the kernel of $A(t_0)$
as $t\in J,\;t\rightarrow t_0$.  By construction of $P$, moreover, $\ker(A(t_0))=\mathrm{span}(P(1,0))=\mathrm{span}(v_\star)$.
Hence, for $\eta$ small enough,
almost every  $t\in \Phi_\eta$
is in $I^+_\eta\cup I^-_\eta$.

To summarize, for $\eta$ small enough, $I^+_\eta\cup I^-_\eta$ has full measure in $(t_0-\eta,t_0)$.
Moreover,
\begin{equation}\label{eq:2lims}
\lim_{t\in I^+_\eta,\;t\rightarrow t_0}u(t)=v_\star,\qquad \lim_{t\in I^-_\eta,\;t\rightarrow t_0}u(t)=-v_\star.
\end{equation}
We next prove
that $u$ converges either to $v_\star$ or to $-v_\star$ as $t\rightarrow t_0$ by
showing that, for $\eta$ small enough, either $I^{+}_\eta$ or $I^{-}_\eta$ has  measure zero.

Suppose by contradiction that there exists a sequence of intervals $(\tau_0^n,\tau_1^n)$ in $(0,t_0)$ such that $\tau_0^n,\tau_1^n\rightarrow t_0$ as $n\to\infty$ and
both $|(\tau_0^n,\tau_1^n)\cap I^+|$  and $|(\tau_0^n,\tau_1^n)\cap I^-|$ are positive, where $|\cdot|$ denotes the Lebesgue measure and $I^\pm=\{t\in (0,t_0)\mid \pm v_\star^Tu(t)>0\}$. Moreover, up to restricting $(\tau_0^n,\tau_1^n)$, we can assume that
\begin{equation}\label{moite-moite}
|(\tau_0^n,\tau_1^n)\cap I^+|=|(\tau_0^n,\tau_1^n)\cap I^-|>0.
\end{equation}
This can be seen, for instance, by considering a continuous deformation of an interval around a
Lebesgue point of $(\tau_0^n,\tau_1^n)\cap I^+$ towards an interval around a
Lebesgue point of $(\tau_0^n,\tau_1^n)\cap I^-$.

For every $n\in\mathbb{N}$, let  $u_n\in L^\infty([0,1],\mathbb{R}^2)$ be defined by $u_n(\tau)=u(\tau_0^n+\tau(\tau_1^n-\tau_0^n))$.
Up to extracting a subsequence, $u_n $  weakly-$\star$  converges to some $u_\star$.
Condition \eqref{moite-moite} and the limits in \eqref{eq:2lims} imply that
\begin{equation}\label{eq:intzero}
\int_0^1 u_\star(t)dt=0.
\end{equation}

Thanks to \eqref{eq:2lims} we also have that
$w_\star^T u_n$
$L^\infty$-converges   to zero  as $n\to \infty$. In particular,
$w_\star^T u_\star\equiv 0$.
By Lemma~\ref{lem:nilpot}, $u_\star$ is optimal and $v_\star^T u_\star$ has values in $\{-1,1\}$. Hence,  by Lemma~\ref{lem:constantsign} in the appendix, $v_\star^T u_\star$ is constantly equal to $+1$ or $-1$.
This contradicts \eqref{eq:intzero} and the proof is concluded.
\qed

\appendix

\section{An elementary lemma}
\label{s:app}

\begin{lemma}\label{lem:constantsign}
Let $(M,\distr,g)$ be a sub-Riemannian manifold.
Let $V$ be a Lipschitz continuous vector field on $T^*M$ such that $\pi_* V(\lambda)\in \distr_{\pi(\lambda)}\setminus\{0\}$ for every $\lambda\in T^*M$.
 Let $\lambda:[0,T]\to T^*M$ satisfy $\dot\lambda(t)=\sigma(t)V(\lambda(t))$
with $\sigma\in L^\infty([0,T],[-1,1])$. Assume that $\gamma=\pi\circ \lambda :[0,T]\to M$ is a length-minimizer. Then $\sigma$ has constant sign, i.e., either $\sigma\ge 0 $ a.e.\ on $[0,T]$ or  $\sigma\le 0$ a.e.\ on $[0,T]$.
\end{lemma}
\begin{proof}
Set $\kappa=\int_0^T\sigma(t)dt$ and notice that $\lambda(T)=e^{\kappa V}(\lambda(0))$.
If $\sigma$ does not have constant sign,
then $[0,1]\ni t\mapsto \pi \circ e^{t \kappa V}(\lambda(0))$ is a curve connecting $\gamma(0)$ to $\gamma(T)$ and having length smaller than $\gamma$.
\end{proof}

A particular case of the lemma occurs when $V=\vec{H}$ is the Hamiltonian vector field on $T^*M$ associated with the Hamiltonian $\lambda\mapsto \langle \lambda, X(\pi (\lambda))\rangle$, where $X$ is  a smooth horizontal never-vanishing vector field on $M$.
This means that if a solution of $\dot\gamma(t)=\sigma(t)X(\gamma(t))$ is a length-minimizer then $\sigma$ has constant sign.

\bibliographystyle{alphaabbr}
\bibliography{biblio}%

\end{document}